\documentclass[10pt, technote, onecolumn]{IEEEtran}
\usepackage{times}
\usepackage{helvet}
\usepackage{courier}
\usepackage{hyperref}
\usepackage[margin=0.5in]{geometry}

\usepackage{amssymb,amsthm,amsmath}
\usepackage{mathtools}
\usepackage{bbm}
\allowdisplaybreaks

\usepackage{tikz}
\usepackage{xifthen}
\usetikzlibrary{calc}

\newtheorem{theorem}{Theorem}
\newtheorem{lemma}{Lemma}
\newtheorem{proposition}{Proposition}

\title{An Improved Lower Bound for the Traveling Salesman Constant}
\author{Julia Gaudio and Patrick Jaillet
\thanks{J. Gaudio and P. Jaillet are with the Massachusetts Institute of Technology.}}
\begin{document}
\maketitle
\begin{abstract}
Let $X_1, X_2, \dots, X_n$ be independent uniform random variables on $[0,1]^2$. Let $L(X_1, \dots, X_n)$ be the length of the shortest Traveling Salesman tour through these points. It is known that there exists a constant $\beta$ such that \[ \lim_{n \to \infty} \frac{L(X_1, \dots, X_n)}{\sqrt{n}} =  \beta \] almost surely (\cite{Beardwood1959}). The original analysis in \cite{Beardwood1959} showed that $\beta \geq 0.625$. Building upon an approach proposed in \cite{Steinerberger2015}, we improve the lower bound to $\beta \geq 0.6277$.
\end{abstract}

\section{Introduction}
 Let $X_1, \dots, X_n$ be independent uniform random variables on $[0,1]^2$. Let $d(x,y) = \Vert x - y \Vert_2$ be the Euclidean distance. Let $L(X_1, \dots, X_n)$ be the distance of the optimal Traveling Salesman tour through these points, under distance $d(\cdot, \cdot)$. In seminal work, Beardwood et al (1959) analyzed the limiting behavior of the value of the optimal Traveling Salesman tour length, under the random Euclidean model.
 
 \begin{theorem}[\cite{Beardwood1959}]
 There exists a constant $\beta$ such that
 \[ \lim_{n \to \infty} \frac{L(X_1, \dots, X_n)}{\sqrt{n}} =  \beta \] almost surely.
 \end{theorem}
 
The authors additionally showed in \cite{Beardwood1959} that
 \[0.625 \leq \beta \leq \beta_+\]
 where $\beta_+ = 2 \int_{0}^{\infty} \int_0^{\sqrt{3}} \sqrt{z_1^2 + z_2^2} e^{-\sqrt{3}z_1} \left(1 - \frac{z_2}{\sqrt{3}}\right) dz_2 dz_1.$ This integral is equal to approximately $0.92116$ (\cite{Steinerberger2015}). To date, the only improvement to the upper bound was given in \cite{Steinerberger2015}, showing that $\beta \leq \beta_+ - \epsilon_0$, for an explicit $\epsilon_0 > \frac{9}{16}10^{-6}$. In \cite{Steinerberger2015}, the author also claimed to improve the lower bound; however, we have found a fault in the argument. 
 
The rest of this note is structured as follows. In Section \ref{sec:approaches}, we present the proof of $\beta \geq 0.625$ by \cite{Beardwood1959}. We then outline the approach of \cite{Steinerberger2015} to improve the bound. Section \ref{sec:lower-bound} corrects the result in \cite{Steinerberger2015}, giving the lower bound $\beta \geq 0.625 + \frac{19}{10368} \approx 0.6268$. Finally, Section \ref{sec:improvement} tightens the argument of \cite{Steinerberger2015} to derive the improved bound, $\beta \geq 0.6277$.

\section{Approaches for the Lower Bound}\label{sec:approaches}
By the following lemma, we can equivalently study the limiting behavior of $\frac{\mathbb{E}\left[L(X_1, \dots, X_n)\right]}{\sqrt{n}}$.
\begin{lemma}[\cite{Beardwood1959}]\label{lemma:expectation}
It holds that
\[\frac{\mathbb{E}\left[L(X_1, \dots, X_n)\right]}{\sqrt{n}} \to \beta.\]
\end{lemma}
Further, we can switch to a Poisson process with intensity $n$. Let $\mathcal{P}_n$ denote a Poisson process with intensity $n$ on $[0,1]^2$.
\begin{lemma}[\cite{Beardwood1959}]\label{lemma:poisson}
It holds that
\[\frac{\mathbb{E}\left[L(\mathcal{P}_n)\right]}{\sqrt{n}} \to \beta.\]
\end{lemma}
\cite{Beardwood1959} gave the following lower bound on $\beta$.
\begin{theorem}[\cite{Beardwood1959}]
The value $\beta$ is lower bounded by $\frac{5}{8}$.
\end{theorem}
\begin{proof} (Sketch)
We outline the proof given by \cite{Beardwood1959}, giving a lower bound on $\mathbb{E}\left[L(\mathcal{P}_n)\right]$. Observe that in a valid traveling salesman tour, every point is connected to exactly two other points. To lower bound, we can connect each point to its two closest points. We can further assume that the Poisson process is over all of $\mathbb{R}^2$, rather than just $[0,1]^2$, in order to remove the boundary effect. The expected distance of a point to its closest neighbor is shown to be $\frac{1}{2\sqrt{n}}$, and the expected distance to the next closes neighbor is shown to be $\frac{3}{4 \sqrt{n}}$. Each point contributes half the expected lengths to the closest two other points. Since the number of points is concentrated around $n$, it holds that $\beta \geq \frac{1}{2} \left( \frac{1}{2} + \frac{3}{4}\right)$.
\end{proof}

Certainly there is room to improve the lower bound. Observe that short cycles are likely to appear when we connect each point to the two closest other points. In \cite{Steinerberger2015}, the author gave an approach to identify situations in which $3$-cycles appear, and then lower-bounded the contribution of correcting these $3$-cycles. We outline the approach below.
\begin{enumerate}
\item For point $a$, let $r_1$ be the distance of $a$ to the closest point, and let $r_2$ be the distance to the next closest point. Let $E_a$ be the event that the third closest point is at a distance of $r_3 \geq r_1 + 2 r_2$.
\item The probability that $E_a$ occurs is calculated to be $\frac{7}{324}$ for a given point $a$. Therefore, the expected number of points satisfying this geometric property is $\frac{7}{324}n$, and the number of triples involved is at least $\frac{1}{3} \frac{7}{324}n$ in expectation.
\item Using the relationship $r_3 \geq r_1 + 2 r_2$, we can show that if $\{a,b,c,d\}$ satisfy the geometric property with $\Vert a - b\Vert = r_1$, $\Vert a - c \Vert = r_2$, and $\Vert a - d \Vert = r_3$, then the closest two points to $b$ are $a$ and $c$, and the closest two points to $c$ are $a$ and $b$. Therefore, the ``count the closest two distances'' method would create a triangle in this situation. 
\item To correct for the triangle, subtract the lengths coming from the triangle and add a lower bound on the new lengths. The triangle contribution is calculated to be at most $3(r_1 + r_2)$ and the new lengths are calculated to be at least $2 r_3$. Therefore, whenever the geometric property holds for a triplet of points, the calculated  contribution is $2r_3 - 3(r_1 + r_2)$.
\item The final adjustment is calculated to be $\frac{19}{5184}$.
\end{enumerate}

There are two errors in this analysis that are both due to inconsistency with counting edge lengths. If edge lengths are counted from the perspective of vertices, then the right thing to do would be to give each vertex two ``stubs.'' These stubs are connected to other vertices, and may form edges if there are agreements. A stub from vertex $a$ to vertex $b$ contributes $\frac{1}{2} \Vert a - b \Vert$ to the path length. In this way, a triangle comprises 6 stubs, and the contribution to the path length is the sum of the edge lengths. On page 35, the author writes $r_1 + r_2 + 2 \Vert a - c\Vert$ as the contribution of the triangle. This is probably a typo and likely $r_1 + r_2 + 2 \Vert b - c\Vert$ was meant instead. However, it should be $r_1 + r_2 +  \Vert b - c\Vert \leq 2 (r_1 + r_2)$. Fixing this error helps the analysis.

The next step is to redirect the six stubs, and determine their length contributions. We break edge $(b,c)$, which means we need to redirect two stubs, while the four stubs that comprise the edges $(a,b)$ and $(a,c)$ remain. The redirected stubs contribute $\frac{1}{2} \Vert b - d \Vert + \frac{1}{2} \Vert c - e \Vert$. The six stubs therefore yield an overall contribution of $\Vert a - b\Vert + \Vert a - c \Vert + \frac{1}{2} \Vert b - d \Vert + \frac{1}{2} \Vert c - e \Vert \geq r_1 + r_2 + \frac{1}{2} \left(r_3 - r_1\right) +  \frac{1}{2} \left(r_3 - r_2\right) = r_3 + \frac{1}{2} (r_1 + r_2)$. In the analysis above Figure 5 in the paper, the author includes the full lengths $\Vert b-d\Vert$ and $\Vert c - e \Vert$. The effect of this is to give points $d$ and $e$ a third stub each.

To summarize, the overall contribution for the triangle scenario, after breaking edges $(b,c)$, is $r_3 + \frac{1}{2} (r_1 + r_2) - 2 (r_1 + r_2) = r_3 - \frac{3}{2} r_1 - \frac{3}{2}r_2$. 


\section{Derivation of the Lower Bound}\label{sec:lower-bound}
In this section we use the approach of \cite{Steinerberger2015} to derive a lower bound on $\beta$.
\begin{theorem}\label{thm:first-lower-bound}
It holds that $\beta \geq \frac{5}{8} + \frac{19}{10368}$.
\end{theorem}
The proof of Theorem \ref{thm:first-lower-bound} requires Lemmas \ref{lemma:density} and  \ref{lemma:integral}.
\begin{lemma}[Lemma 4 in \cite{Steinerberger2015}]\label{lemma:density}
Let $\mathcal{P}_n$ be a Poisson point process on $\mathbb{R}^2$ with intensity $n$. Then for any fixed point $p \in \mathbb{R}^2$, the probability distribution of the distance between $p$ and the the three closest points to $p$ is given by
\begin{align*}
h(r_1,r_2,r_3) &= \begin{cases}
e^{-n \pi r_3^3} (2n\pi)^3 r_1 r_2 r_3 & \text{if } r_1 < r_2 < r_3\\
0 & \text{otherwise.}
\end{cases}
\end{align*}
\end{lemma}
\begin{lemma}\label{lemma:integral}
\begin{align*}
\int_{r_1 = 0}^{\infty} \int_{r_2 = r_1}^{\infty} \int_{r_3 = r_1 + 2 r_2}^{\infty} \left(r_3 - \frac{3}{2}r_1 - \frac{3}{2}r_2 \right) e^{-n \pi r_3^2} r_1 r_2 r_3 dr_3 dr_2 dr_1&= \frac{19}{27648 \pi^3 n^{\frac{7}{2}}} 
\end{align*}
\end{lemma}

\begin{proof}[Proof of Theorem \ref{thm:first-lower-bound}]
First we verify that the lower bound from breaking edge $(b,c)$ is valid. If edge $(a,b)$ is broken instead, the new stub lengths are $\Vert a - c \Vert + \Vert b - c \Vert + \frac{1}{2} \Vert a - d \Vert + \frac{1}{2} \Vert b - e \Vert$. The difference after subtracting the original stub lengths is then equal to 
\vspace{-6pt}
\begin{align*}
\Vert a - c \Vert + \Vert b - c \Vert + \frac{1}{2} \Vert a - d \Vert + \frac{1}{2} \Vert b - e \Vert - \left(\Vert a - c \Vert + \Vert b - c \Vert + \Vert a - b\Vert \right) &=  \frac{1}{2} \Vert a - d \Vert + \frac{1}{2} \Vert b - e \Vert -  \Vert a - b\Vert \\
&\geq \frac{1}{2} r_3 + \frac{1}{2} \left( \Vert a - e \Vert - \Vert a - b \Vert \right) - r_1\\
&\geq \frac{1}{2} r_3 + \frac{1}{2} \left( r_3 - r_1 \right) - r_1\\
&= r_3 - \frac{3}{2}r_1
\end{align*}
Similarly, if edge $(a,c)$ is broken, the contribution is lower bounded by $r_3 - \frac{3}{2} r_2$. Since $r_3 - \frac{3}{2} r_1 - \frac{3}{2} r_2 \leq r_3 - \frac{3}{2} r_2 \leq r_3 - \frac{3}{2} r_1$, we conclude that $r_3 - \frac{3}{2} r_1 - \frac{3}{2} r_1$ from breaking edge $(b,c)$ is a valid lower bound.

Therefore, from the discussion in Section \ref{sec:approaches} and Lemma \ref{lemma:density} we adjust the integral in \cite{Steinerberger2015} to give
\begin{align*}
\beta \geq \frac{5}{8} + \frac{\sqrt{n}}{3} \int_{r_1 = 0}^{\infty} \int_{r_2 = r_1}^{\infty} \int_{r_3 = r_1 + 2r_2}^{\infty} \left(r_3 - \frac{3}{2}r_1 - \frac{3}{2}r_2 \right) e^{-n \pi r_3^2} (2n\pi)^3 r_1 r_2 r_3 dr_3 dr_2 dr_1
\end{align*}
From Lemma \ref{lemma:integral}, 
\[\beta \geq \frac{5}{8} + \frac{\sqrt{n}}{3} (2n\pi)^3 \frac{19}{27648 \pi^3 n^{\frac{7}{2}}} = \frac{5}{8} + \frac{19}{10 368} \approx 0.626833.\]
\end{proof}

\section{An Improvement}\label{sec:improvement}
In this section, we improve upon the bound in Section \ref{sec:lower-bound} by tightening the triangle inequality.
\begin{theorem}\label{thm:second-lower-bound}
It holds that \begin{align*}
\beta &\geq \frac{5}{8} + \frac{1}{2} \left(\frac{19}{10368}\right) + \frac{1}{2}\left(\frac{3072 \sqrt{2} - 4325}{5376}\right)\\
&\geq 0.6277.
\end{align*}
\end{theorem}
\begin{proof}
Place a Cartesian grid so that point $a$ is at the origin and point $b$ is at $(r_1, 0)$. Then with probability $\frac{1}{2}$, point $c$ falls into the first or fourth quadrant, and with probability $\frac{1}{2}$, point $c$ falls into the second or third quadrant. Conditioned on point $c$ falling into the first or fourth quadrant, the maximum length of $\Vert b - c\Vert$ is $\sqrt{r_1^2 + r_2^2}$. Conditioned on point $c$ falling into the second or third quadrant, the maximum length of $\Vert b - c\Vert$ is $r_1 + r_2$, which corresponds to the computation in Section \ref{sec:lower-bound}.

Conditioned on point $c$ falling into the first or fourth coordinate, the length contribution from breaking edge $(b,c)$ is at least $r_3 + \frac{1}{2} \left( r_1 + r_2\right) - \left(r_1 + r_2 + \sqrt{r_1^2 + r_2^2}\right) = r_3 - \frac{1}{2} r_1  - \frac{1}{2} r_2 - \sqrt{r_1^2 + r_2^2}$. If edge $(a,b)$ is broken instead, the new stub lengths are $\Vert a - c \Vert + \Vert b - c \Vert + \frac{1}{2} \Vert a - d \Vert + \frac{1}{2} \Vert b - e \Vert$. The difference after subtracting the original stub lengths is then equal to 
\vspace{-6pt}
\begin{align*}
\Vert a - c \Vert + \Vert b - c \Vert + \frac{1}{2} \Vert a - d \Vert + \frac{1}{2} \Vert b - e \Vert - \left(\Vert a - c \Vert + \Vert b - c \Vert + \Vert a - b\Vert \right) &=  \frac{1}{2} \Vert a - d \Vert + \frac{1}{2} \Vert b - e \Vert -  \Vert a - b\Vert \\
&\geq \frac{1}{2} r_3 + \frac{1}{2} \left( \Vert a - e \Vert - \Vert a - b \Vert \right) - r_1\\
&\geq \frac{1}{2} r_3 + \frac{1}{2} \left( r_3 - r_1 \right) - r_1\\
&= r_3 - \frac{3}{2}r_1
\end{align*}

\vspace{-6pt}
Similarly, if edge $(a,c)$ is broken, the contribution is lower bounded by $r_3 - \frac{3}{2} r_2$. Since $r_3 - \frac{1}{2} r_1 - \frac{1}{2}r_2 - \sqrt{r_1^2 + r_2^2} \leq r_3 - \frac{3}{2} r_2 \leq r_3 - \frac{3}{2} r_1$, we conclude that $r_3 - \frac{1}{2} r_1 - \frac{1}{2}r_2 - \sqrt{r_1^2 + r_2^2}$ from breaking edge $(b,c)$ is a valid lower bound.

We therefore break edge $(b,c)$. 
\begin{proposition}
If $r_3 \geq r_2 + \sqrt{r_1^2 + r_2^2}$, then the closest points to each of $a,b,c$ are the other two points in the set $\{a,b,c\}$, whenever point $b$ is in the first or fourth quadrant. 
\end{proposition}
\begin{proof}
Point $a$: $d(a,b) = r_1$, $d(a,c) = r_2$, and for any $d \notin\{a,b,c\}$, it holds that $d(a,d) \geq r_3 \geq r_2 + \sqrt{r_1^2 + r_2^2}$. Therefore $d(a,d) \geq d(a,b)$ and $d(a,d) \geq d(a,c)$.\\
Point $b$: $d(a,b) = r_1$, $d(b,c) \leq \sqrt{r_1^2 + r_2^2}$, and for any $d \notin\{a,b,c\}$, it holds that $d(b,d) \geq d(a,d) - d(a,b) \geq r_2 + \sqrt{r_1^2 + r_2^2} - r_1$. Therefore $d(b,d) \geq d(a,b)$ and $d(b,d) \geq d(b,c)$.\\
Point $c$: $d(a,c) = r_2$, $d(b,c) \leq \sqrt{r_1^2 + r_2^2}$, and for any $d \notin\{a,b,c\}$, it holds that $d(c,d) \geq d(a,d) - d(a,c) \geq r_2 + \sqrt{r_1^2 + r_2^2} - r_2 = \sqrt{r_1^2 + r_2^2}$. Therefore $d(c,d) \geq d(a,c)$ and $d(c,d) \geq d(b,c)$.
\end{proof}

The lower bound on $\beta$ is therefore
\begin{align*}
\frac{5}{8} + \frac{\sqrt{n}}{3} \int_{r_1 = 0}^{\infty} \int_{r_2 = r_1}^{\infty} \int_{r_3 = r_2 + \sqrt{r_1^2 + r_2^2}}^{\infty} \left(r_3 - \frac{1}{2}r_1 - \frac{1}{2}r_2 - \sqrt{r_1^2 + r_2^2} \right) e^{-n \pi r_3^2} (2n\pi)^3 r_1 r_2 r_3 dr_3 dr_2 dr_1.
\end{align*}
\begin{lemma}\label{lemma:second-integral}
It holds that
\small
\begin{align*}
&\int_{r_1 = 0}^{\infty} \int_{r_2 = r_1}^{\infty} \int_{r_3 = r_2 + \sqrt{r_1^2 + r_2^2}}^{\infty} \left(r_3 - \frac{1}{2}r_1 - \frac{1}{2}r_2 - \sqrt{r_1^2 + r_2^2} \right) e^{-n \pi r_3^2}  r_1 r_2 r_3 dr_3 dr_2 dr_1 \\
&=  \left[- \frac{\left(\frac{1}{1+\sqrt{2}} \right)^8}{8 \cdot 48} - \frac{\left(\frac{1}{1+\sqrt{2}} \right)^7}{7 \cdot 16} - \frac{\left(\frac{1}{1+\sqrt{2}} \right)^6}{6 \cdot 16} + \frac{1}{5}\left(\frac{1}{8} + \frac{1}{4} + \frac{1}{6} + \frac{2^{\frac{3}{2}}}{3} \right)\left(\frac{1}{1+\sqrt{2}} \right)^5 - \frac{13\left(\frac{1}{1+\sqrt{2}} \right)^4 }{64} - \frac{\left(\frac{1}{1+\sqrt{2}} \right)^3}{48} + \frac{\left(\frac{1}{1+\sqrt{2}} \right)^2}{32}  \right] \frac{15}{16\pi^3 n^{\frac{7}{2}}}.
\end{align*}
\normalsize
\end{lemma}

Multiplying the value of the integral in Lemma \ref{lemma:second-integral} by $\frac{\sqrt{n}(2 n \pi)^3}{3}$, we obtain the following lower bound.
\begin{align*}
&\frac{5}{8} +  \frac{5}{2} \left[- \frac{\left(\frac{1}{1+\sqrt{2}} \right)^8}{8 \cdot 48} - \frac{\left(\frac{1}{1+\sqrt{2}} \right)^7}{7 \cdot 16} - \frac{\left(\frac{1}{1+\sqrt{2}} \right)^6}{6 \cdot 16} + \frac{1}{5}\left(\frac{1}{8} + \frac{1}{4} + \frac{1}{6} + \frac{2^{\frac{3}{2}}}{3} \right)\left(\frac{1}{1+\sqrt{2}} \right)^5 - \frac{13\left(\frac{r_3}{1+\sqrt{2}} \right)^4 }{64} - \frac{\left(\frac{1}{1+\sqrt{2}} \right)^3}{48} + \frac{\left(\frac{1}{1+\sqrt{2}} \right)^2}{32}  \right] \\
&= \frac{5}{8} +  \frac{3072 \sqrt{2} - 4325}{5376}\\
&\approx \frac{5}{8} + 0.003621
\end{align*}
Finally, conditioning on the quadrant, the overall lower bound is
\begin{align*}
\beta \geq \frac{5}{8} + \frac{1}{2} \left(\frac{19}{10368}\right) + \frac{1}{2}\left(\frac{3072 \sqrt{2} - 4325}{5376}\right) \geq 0.6277
\end{align*}
\end{proof}

\section*{Appendix}
\begin{proof}[Proof of Lemma \ref{lemma:integral}]
We can change the order of integration to compute the integral more easily. 
\begin{align*}
&\int_{r_1 = 0}^{\infty} \int_{r_2 = r_1}^{\infty} \int_{r_3 = r_1 + 2 r_2}^{\infty} \left(r_3 - \frac{3}{2}r_1 - \frac{3}{2}r_2 \right) e^{-n \pi r_3^2} r_1 r_2 r_3 dr_3 dr_2 dr_1\\
&=\int_{r_3 = 0}^{\infty} \int_{r_1 = 0}^{\frac{r_3}{3}} \int_{r_2 = r_1}^{\frac{r_3 - r_1}{2}} \left(r_3 - \frac{3}{2}r_1 - \frac{3}{2}r_2 \right) e^{-n \pi r_3^2} r_1 r_2 r_3 dr_2 dr_1 dr_3\\
&=\int_{r_3 = 0}^{\infty}r_3 e^{-n \pi r_3^2} \int_{r_1 = 0}^{\frac{r_3}{3}} r_1 \int_{r_2 = r_1}^{\frac{r_3 - r_1}{2}} r_2 \left(r_3 - \frac{3}{2}r_1 - \frac{3}{2}r_2 \right)  dr_2 dr_1 dr_3\\
&=\int_{r_3 = 0}^{\infty} r_3 e^{-n \pi r_3^2}\int_{r_1 = 0}^{\frac{r_3}{3}}r_1 \left( \frac{r_2^2}{2} \left(r_3 -\frac{3}{2}r_1\right) - \frac{1}{2} r_2^3 \right) \Big |_{r_2=r_1}^{\frac{r_3-r_1}{2}} dr_1 dr_3\\
&=\int_{r_3 = 0}^{\infty} r_3 e^{-n \pi r_3^2}\int_{r_1 = 0}^{\frac{r_3}{3}}r_1 \left( \frac{\left(\frac{r_3-r_1}{2}\right)^2 - r_1^2}{2} \left(r_3 -\frac{3}{2}r_1\right) - \frac{1}{2} \left( \left(\frac{r_3-r_1}{2}\right)^3 - r_1^3 \right) \right)  dr_1 dr_3\\
&=\int_{r_3 = 0}^{\infty} r_3 e^{-n \pi r_3^2}\int_{r_1 = 0}^{\frac{r_3}{3}} \left(\frac{9 r_1^4}{8} - \frac{3r_1^3 r_3}{16} - \frac{r_1^2 r_3^2}{4} + \frac{r_1 r_3^3}{16} \right)  dr_1 dr_3\\
&=\int_{r_3 = 0}^{\infty} r_3 e^{-n \pi r_3^2} \left(\frac{9 r_1^5}{40} - \frac{3r_1^4 r_3}{64} - \frac{r_1^3 r_3^2}{12} + \frac{r_1^2 r_3^3}{32} \right) \Big |_{r_1 = 0}^{\frac{r_3}{3}}  dr_3\\
&=\int_{r_3 = 0}^{\infty} r_3 e^{-n \pi r_3^2} \left(\frac{9 \left(\frac{r_3}{3}\right)^5}{40} - \frac{3\left(\frac{r_3}{3}\right)^4 r_3}{64} - \frac{\left(\frac{r_3}{3}\right)^3 r_3^2}{12} + \frac{\left(\frac{r_3}{3}\right)^2 r_3^3}{32} \right) dr_3\\
&= \left(\frac{9 \left(\frac{1}{3}\right)^5}{40} - \frac{3\left(\frac{1}{3}\right)^4}{64} - \frac{\left(\frac{1}{3}\right)^3}{12} + \frac{\left(\frac{1}{3}\right)^2}{32} \right) \int_{r_3 = 0}^{\infty} r_3^6 e^{-n \pi r_3^2}  dr_3\\
&= \frac{19}{25920} \int_{r_3 = 0}^{\infty} r_3^6 e^{-n \pi r_3^2}  dr_3\\
&= \frac{19}{25920} \frac{15}{16 \pi^3 n^{\frac{7}{2}}} \\
&= \frac{19}{27648 \pi^3 n^{\frac{7}{2}}} 
\end{align*}
\end{proof}

\begin{proof}[Proof of Lemma \ref{lemma:second-integral}]
Again we change the order of integration to compute the integral more easily. 

Given $r_3$, the upper bound on $r_1$ is derived by setting $r_3 = r_1  + \sqrt{2r_1^2} \iff r_1 = \frac{r_3}{1 + \sqrt{2}}$. Given $r_3$ and $r_1$, set $r_3 = r_2 + \sqrt{r_1^2 + r_2^2}$. We have
\begin{align*}
&\left(r_3 - r_2\right)^2 = r_1^2 + r_2^2\\
&r_3^2 - 2r_2 r_3 + r_2^2 = r_1^2 + r_2^2\\
&r_3^2 - 2r_2 r_3  = r_1^2\\
&r_2 = \frac{r_3^2 - r_1^2}{2r_3}
\end{align*}
Therefore,
\scriptsize
\begin{align*}
&\int_{r_1 = 0}^{\infty} \int_{r_2 = r_1}^{\infty} \int_{r_3 = r_2 + \sqrt{r_1^2 + r_2^2}}^{\infty} \left(r_3 - \frac{1}{2}r_1 - \frac{1}{2}r_2 - \sqrt{r_1^2 + r_2^2} \right) e^{-n \pi r_3^2}  r_1 r_2 r_3 dr_3 dr_2 dr_1\\
&= \int_{r_3 = 0}^{\infty} r_3 e^{-n \pi r_3^2}\int_{r_1 = 0}^{\frac{r_3}{1 + \sqrt{2}}} r_1 \int_{r_2 = r_1}^{\frac{r_3^2 - r_1^2}{2r_3}} r_2\left(r_3 - \frac{1}{2}r_1 - \frac{1}{2}r_2 - \sqrt{r_1^2 + r_2^2} \right)   dr_2 dr_1 dr_3\\
&= \int_{r_3 = 0}^{\infty} r_3 e^{-n \pi r_3^2}\int_{r_1 = 0}^{\frac{r_3}{1 + \sqrt{2}}} r_1 \left[\frac{r_2^2}{2} \left( r_3 - \frac{1}{2} r_1 \right) - \frac{1}{6} r_2^3 - \frac{1}{3} \left(r_1^2 + r_2^2 \right)^{\frac{3}{2}}\right]   \Big |_{r_2 = r_1}^{\frac{r_3^2 - r_1^2}{2r_3}} dr_1 dr_3\\
&= \int_{r_3 = 0}^{\infty} r_3 e^{-n \pi r_3^2}\int_{r_1 = 0}^{\frac{r_3}{1 + \sqrt{2}}} r_1 \left[\frac{\left(\frac{r_3^2 - r_1^2}{2r_3} \right)^2}{2} \left( r_3 - \frac{1}{2} r_1 \right) - \frac{1}{6} \left(\frac{r_3^2 - r_1^2}{2r_3} \right)^3 - \frac{1}{3} \left(r_1^2 + \left(\frac{r_3^2 - r_1^2}{2r_3} \right)^2 \right)^{\frac{3}{2}} - \frac{r_1^2}{2} \left( r_3 - \frac{1}{2} r_1 \right) + \frac{1}{6} r_1^3 + \frac{1}{3} \left(r_1^2 + r_1^2 \right)^{\frac{3}{2}}\right]   dr_1 dr_3\\
&= \int_{r_3 = 0}^{\infty} r_3 e^{-n \pi r_3^2}\int_{r_1 = 0}^{\frac{r_3}{1 + \sqrt{2}}} r_1 \left[\frac{\left(\frac{r_3^2 - r_1^2}{2r_3} \right)^2}{2} \left( r_3 - \frac{1}{2} r_1 \right) - \frac{1}{6} \left(\frac{r_3^2 - r_1^2}{2r_3} \right)^3 - \frac{1}{3} \left(\frac{\left(r_1^2 + r_3^2 \right)^2}{4r_3^2} \right)^{\frac{3}{2}} - \frac{r_1^2 r_3}{2} + \left(\frac{1}{4} + \frac{1}{6} + \frac{2^{\frac{3}{2}}}{3} \right) r_1^3\right]   dr_1 dr_3\\
&= \int_{r_3 = 0}^{\infty} r_3 e^{-n \pi r_3^2}\int_{r_1 = 0}^{\frac{r_3}{1 + \sqrt{2}}} r_1 \left[\frac{\left(\frac{r_3^2 - r_1^2}{2r_3} \right)^2}{2} \left( r_3 - \frac{1}{2} r_1 \right) - \frac{1}{6} \left(\frac{r_3^2 - r_1^2}{2r_3} \right)^3 - \frac{1}{3} \left(\frac{r_1^2 + r_3^2 }{2r_3} \right)^3 - \frac{r_1^2 r_3}{2} + \left(\frac{1}{4} + \frac{1}{6} + \frac{2^{\frac{3}{2}}}{3} \right) r_1^3\right]   dr_1 dr_3\\
&= \int_{r_3 = 0}^{\infty} r_3 e^{-n \pi r_3^2}\int_{r_1 = 0}^{\frac{r_3}{1 + \sqrt{2}}}  \left[- \frac{r_1^7}{48 r_3^3} - \frac{r_1^6}{16r_3^2} - \frac{r_1^5}{16r_3} + \left(\frac{1}{8} + \frac{1}{4} + \frac{1}{6} + \frac{2^{\frac{3}{2}}}{3} \right)r_1^4 - \frac{13r_1^3 r_3}{16} - \frac{r_1^2 r_3^2}{16} + \frac{r_1 r_3^3}{16}  \right]   dr_1 dr_3\\
&= \int_{r_3 = 0}^{\infty} r_3 e^{-n \pi r_3^2} \left[- \frac{r_1^8}{8 \cdot 48 r_3^3} - \frac{r_1^7}{7 \cdot 16r_3^2} - \frac{r_1^6}{6 \cdot 16r_3} + \frac{1}{5}\left(\frac{1}{8} + \frac{1}{4} + \frac{1}{6} + \frac{2^{\frac{3}{2}}}{3} \right)r_1^5 - \frac{13r_1^4 r_3}{64} - \frac{r_1^3 r_3^2}{48} + \frac{r_1^2 r_3^3}{32}  \right]   \Big |_{r_1 = 0}^{\frac{r_3}{1 + \sqrt{2}}} dr_3\\
&= \int_{r_3 = 0}^{\infty} r_3 e^{-n \pi r_3^2} \left[- \frac{\left(\frac{r_3}{1+\sqrt{2}} \right)^8}{8 \cdot 48 r_3^3} - \frac{\left(\frac{r_3}{1+\sqrt{2}} \right)^7}{7 \cdot 16r_3^2} - \frac{\left(\frac{r_3}{1+\sqrt{2}} \right)^6}{6 \cdot 16r_3} + \frac{1}{5}\left(\frac{1}{8} + \frac{1}{4} + \frac{1}{6} + \frac{2^{\frac{3}{2}}}{3} \right)\left(\frac{r_3}{1+\sqrt{2}} \right)^5 - \frac{13\left(\frac{r_3}{1+\sqrt{2}} \right)^4 r_3}{64} - \frac{\left(\frac{r_3}{1+\sqrt{2}} \right)^3 r_3^2}{48} + \frac{\left(\frac{r_3}{1+\sqrt{2}} \right)^2 r_3^3}{32}  \right]   dr_3\\
&= \left[- \frac{\left(\frac{1}{1+\sqrt{2}} \right)^8}{8 \cdot 48} - \frac{\left(\frac{1}{1+\sqrt{2}} \right)^7}{7 \cdot 16} - \frac{\left(\frac{1}{1+\sqrt{2}} \right)^6}{6 \cdot 16} + \frac{1}{5}\left(\frac{1}{8} + \frac{1}{4} + \frac{1}{6} + \frac{2^{\frac{3}{2}}}{3} \right)\left(\frac{1}{1+\sqrt{2}} \right)^5 - \frac{13\left(\frac{1}{1+\sqrt{2}} \right)^4 }{64} - \frac{\left(\frac{1}{1+\sqrt{2}} \right)^3}{48} + \frac{\left(\frac{1}{1+\sqrt{2}} \right)^2}{32}  \right] \int_{r_3 = 0}^{\infty} r_3^6 e^{-n \pi r_3^2} dr_3  \\
&= \left[- \frac{\left(\frac{1}{1+\sqrt{2}} \right)^8}{8 \cdot 48} - \frac{\left(\frac{1}{1+\sqrt{2}} \right)^7}{7 \cdot 16} - \frac{\left(\frac{1}{1+\sqrt{2}} \right)^6}{6 \cdot 16} + \frac{1}{5}\left(\frac{1}{8} + \frac{1}{4} + \frac{1}{6} + \frac{2^{\frac{3}{2}}}{3} \right)\left(\frac{1}{1+\sqrt{2}} \right)^5 - \frac{13\left(\frac{1}{1+\sqrt{2}} \right)^4 }{64} - \frac{\left(\frac{1}{1+\sqrt{2}} \right)^3}{48} + \frac{\left(\frac{1}{1+\sqrt{2}} \right)^2}{32}  \right] \frac{15}{16\pi^3 n^{\frac{7}{2}}} 
\end{align*}
\end{proof}
\normalsize

\bibliographystyle{plain}
\bibliography{../MIT_Project_5}
\end{document}